\documentclass{amsart}

\usepackage[utf8]{inputenc}
\usepackage[T1]{fontenc}
\usepackage{verbatim}
\usepackage{graphicx}
\usepackage{graphicx,caption2,psfrag,float,color}
\usepackage{amssymb}
\usepackage{amscd}
\usepackage{amsmath}
\usepackage{enumerate}

\usepackage{mathrsfs}

\usepackage[T1]{fontenc}

\newtheorem{theorem}{Theorem}[section]

\newtheorem{lemma}[theorem]{Lemma}

\numberwithin{equation}{subsection}
\newtheorem{definition}[theorem]{Definition}

\title{The behaviour of a certain additive function in large intervals between consecutive primes}
\author{Michael Th. Rassias}
\date{\today}
\address{Department of Mathematics and Engineering Sciences,
Hellenic Military Academy,
16673 Vari Attikis, Greece }\thanks{}
\begin{document}

 \maketitle
 
\begin{abstract} 
We investigate the behaviour of a certain additive function depending on prime divisors of specific integers lying in large gaps between consecutive primes. The result is obtained by a combination of results and ideas related to large gaps between primes and the Erd\"os-Kac theorem, especially the Kubilius model from Probabilistic Number Theory.\\
  
\newline
\end{abstract}

\section{Introduction}

This paper merges two themes:\\
(i) The occurence of large gaps between consecutive primes\\
and\\
(ii) the behaviour of additive arithmetic functions on special sets of integers.\\

Let $p_1=2<p_2<\cdots < p_n$ be the sequence of prime numbers, $d_n=p_{n+1}-p_n$ the $n$-th gap between consecutive primes.\\
The gap $d_n$ is large, when it is large in comparison to the average defined by
$$A(x)=x^{-1} \sum_{n\leq x} \frac{d_n}{\log p_n}.$$
It is a simple consequence of the Prime Number Theorem that
$$\lim_{x\rightarrow \infty} A(x)=1$$
and thus
\[
\limsup_{x\rightarrow \infty} \frac{d_n}{\log p_n}\geq 1.    \tag{1.1}
\]
By introducing the function 
$$G(x)=\max_{n\leq x}\frac{d_n}{\log p_n}$$
(1.1) can also be stated in the form
\[
G(x)\geq 1. \tag{1.2}
\]
We present below a short history of large gap results:\\
In 1931, Westzynthius \cite{Westzynthius}, improving on prior results of B\"acklund \cite{backlund} and Brauer-Zeitz \cite{brauer} proved that 
\[
G(x)\gg \frac{\log x \log_3 x}{\log_4 x}. \tag{1.3}
\]
(Here and in the sequel we define $\log_k x$ by $\log_1 x=\log x$ and $\log_k x = \log(\log_{k-1} x)$, $(k>1$)).\\
Erd\H{o}s \cite{erdos} sharpened (1.3) to 
\[
G(x)\gg \frac{\log x \log_2 x}{(\log_3 x)^2} \tag{1.4}
\]
and in 1938 Rankin \cite{rankin} made a subsequent improvement, namely
\[ 
G(x)\gg \frac{\log x\log_2 x \log_4 x}{(\log_3 x)^2}. \tag{1.5}
\]
His method differs only slightly from that of Erd\H{o}s. Starting  with the paper \cite{erdos} of Erd\H{o}s, all the results on large gaps between primes are baseed on modifications of the Erd\H{o}s-Rankin method.\\
We shortly describe its main features: Let $x>1$. All steps are considered for $x\rightarrow \infty$. Let
\[
P(x)=\prod_{p<x} p, \ y>x.  \tag{1.6}
\]
By the Prime Number Theorem we have 
$$P(x)=e^{x(1+o(1))}\:.$$
A system of congruence classes:
\begin{align*}
\{v\::\: v&\equiv h_{p_1} \bmod p_1\}\\
&\ \vdots \tag{1.7} \\
\{v\::\: v&\equiv h_{p_l} \bmod p_l\}
\end{align*}
is constructed, such that the congruence classes 
$$h_{p_l}\bmod p_l$$ 
cover the interval $(0, y]$.\\
Another system of congruences, closely linked to (1.7) is of crucial importance. The system
\begin{align*}
m&\equiv -h_{p_1} \bmod p_1\\ 
&\ \vdots \tag{1.8}\\  
m&\equiv -h_{p_l} \bmod p_l.
\end{align*}
By the Chinese Remainder Theorem, the system (1.8) has a unique solution $m_0$, with
$$1\leq m_0< P(x).$$
Let $v\in \mathbb{N}$, $1\leq v<y$. Then there is $j$, $1\leq j\leq l$, such that 
$$v\equiv h_{p_j} \bmod p_j.$$
From (1.6), (1.7) and (1.8)
$$m_0+v\equiv 0 \bmod p_j.$$
If $m_0$ is sufficiently large, then all integers $w\in (m_0, m_0+y)$ are composite. If
\[
p_n=\max\{p\ \text{prime}\:,\ p\leq m_0\}\tag{1.9}
\]
then it follows that 
$$p_{n+1}-p_n\geq y,$$
a large gap result.\\
After a few sieving steps, one obtains a residual set $\mathcal{R}_j$ consisting of the union of a set of prime numbers
\[
Q=\{ q\ \text{prime}\::\: x< q\leq y\} \tag{1.10}
\]
and a set of $z$-smooth integers, i.e. integers whose largest prime factor is $\leq z$, where $z$ is chosen appropriately. The fact that the number of smooth integers is very small is crucial. Thus the construction of the last sieving steps is basically reduced to the problem of the choice of residue classes \mbox{$h_p \bmod p$ } that cover the set $Q$ of primes in $(1, y]$.\\
An important quantity is the hitting number of the sieving step covering the set $Q$ of primes in $(1, y]$. For each prime $p$ used  in the sieving step it is defined as the number of elements belonging to the congruence class $h_p \bmod p$. In all papers prior to \cite{maierpom} this hitting number was bounded below by 1. Thus for each element $u$ of the residual set $\mathcal{R}_j$ a prime $p(u)$ could be found and the removal of at least one element from the residue class $h(p(u))\bmod p(u)$ was guaranteed.\\
In the paper \cite{maierpom} by Maier-Pomerance for a positive proportion of primes $p\in \mathcal{S}_3$ a hitting number of at least 2 could be achieved. The pairs of primes in the congruence classes $h_p \bmod p$ can be seen as generalized twin primes and the result needed could be obtained by application of the circle method. Problems on intersections of different arithmetic progressions need to be resolved. This was done by the use of a graph with the primes as vertices and the arithmetic progressions as edges. These combinatorial arguments were further improved by Pintz \cite{pintz} who obtained a hitting number of 2 for all primes $p\in \mathcal{S}_3$.\\
It was a famous prize problem of Erd\H{o}s to improve on the order of magnitude of the lower bound for $G(x)$ in (1.4). This could be achieved after more than 70 years in the papers \cite{ford} and \cite{ford2}, where the hitting number was tending to infinity together with $x$. A crucial ingredient was a result of Maynard \cite{maynard}. This result in turn is related to the great breakthrough results on small gaps between consecutive primes based on the GPY sieve - named after Goldston, Pintz and Yildirim - (see \cite{GPYI, GPYII}).\\
In the paper \cite{ford2}, Ford, Green, Konyagin, Maynard and Tao prove\\

\noindent\textit{Theorem 1 of \cite{ford2}:\\
For any sufficiently large $X$, one has }
\[
G(X)\gg \frac{\log X \log_2 X \log_4 X}{\log_3 X}. \tag{1.11}
\]
The implied constant is effective.\\
The proof combines ideas from the paper \cite{maynard} and the generalization of a hypergraph covering theorem of Pippenger and Spencer \cite{pip}.\\
In a sequence of several sieving steps, the problem is reduced to a problem of sieving a set $Q$ of primes in $[y]\setminus [x]$, which we now describe:\\
For $x$ sufficiently large, define
\[
y=C_0\:x\: \frac{\log x\log_3 x}{\log_2 x}. \tag{1.12}
\]
(Here and in the sequel the $C_i$'s denote fixed positive constants).\\
Let $z=x^{\log_3 x / (4\log_2 x)}$ and introduce three disjoint sets of primes
$$S=\{ s\ \text{prime}\::\: \log^{20} x< s\leq z\}\:,$$
\[
P=\{ p\ \text{prime}\::\: x/2< p\leq x\}\:,  \tag{1.13}
\]
\[
Q=\{ q\ \text{prime}\::\: x< q\leq y\}\:.
\]
For residue classes 
$$\vec{a}=\{ m\in\mathbb{Z}\::\: m\not\equiv a_s(\bmod s)\ \text{for all}\ s\in S\}$$
and likewise
$$\vec{b}=(b_p\bmod p)_{p\in P},$$
define the sifted sets
\[
S(\vec{a})=\{ m\in\mathbb{Z}\::\: m\not\equiv a_s(\bmod s)\ \text{for all}\ s\in S\}  \tag{1.14}
\]
and likewise
\[
S(\vec{b})=\{ m\in\mathbb{Z}\::\: m\not\equiv b_p(\bmod p)\ \text{for all}\ p\in P\}.  \tag{1.15}
\]
The crucial sieving result is\\

\noindent\textit{Theorem 2 of \cite{ford2}:\\
Let $x$ be sufficiently large and suppose that $y$ obeys (1.12). Then there are vectors
$$\vec{a}=(a_s\bmod s)_{s\in S}\ \ \text{and}\ \ \vec{b}=(b_p \bmod p)_{p\in P}\:, $$
such that }
\[
\# (Q\cap S(\vec{a})\cap S(\vec{b})) \ll \frac{x}{\log x}. \tag{1.16}
\]
In \cite{mr2} the author in joint work with H. Maier treated the problem of \textit{Large gaps between consecutive prime numbers containing perfect $k$-th powers of prime numbers}. He obtains the following result\\

\noindent\textit{Theorem 1.1 of \cite{mr2}:\\
There is a constant $C_1>0$ and infinitely many $n$, such that 
$$p_{n+1}-p_n\geq C_1\: \frac{\log p_n \log_2 p_n \log_4 p_n}{\log_3 p_n}$$
and the interval $(p_n, p_{n+1}]$ contains the $k$-th power of a prime.
}\\

\noindent The proof consists of a combination of the method of \cite{ford2}, the matrix method of \mbox{H. Maier} and the method of the paper \cite{konyagin} of Ford, Heath- Brown and Konyagin. The sieving steps (1.14) and (1.15) are modified by restricting the residue classes used for sieving.\\
One defines
\begin{align*}\tag{1.17}
\mathcal{A}=&\{ \vec{a}= (a_s \bmod s)_{s\in S}\::\: \exists\: c_s\ \text{such that}\\  
&\ \ a_s\equiv 1-(c_s+1)^k(\bmod s),\ c_s\not\equiv -1 (\bmod s)\},  
\end{align*}
\begin{align*}\tag{1.18}
\mathcal{B}=&\{ \vec{b}= (b_p \bmod p)_{p\in P}\::\: \exists\: d_p\ \text{such that}\\  
&\ \ b_p\equiv 1-(d_p+1)^k(\bmod p),\ d_p\not\equiv -1 (\bmod p)\}.  
\end{align*}

For $\vec{a}=(a_s \bmod s)_{s\in S}$ and $\vec{b}=(b_p \bmod p)_{p\in P}$ one defines the sifted sets
\[
S(\vec{a})=\{m\in\mathbb{Z}\::\: m\not\equiv a_s(\bmod s)\ \text{for all}\ s\in S\} \tag{1.19}
\]
and 
\[
S(\vec{b})=\{m\in\mathbb{Z}\::\: m\not\equiv b_p(\bmod p)\ \text{for all}\ p\in P\}. \tag{1.20}
\]
One obtains the modification of Theorem 2 (sieving primes) of \cite{ford2}:\\

\noindent\textit{Theorem (3.1) of \cite{mr2}:\\
Let $x$ be sufficiently large and suppose that $y$ obeys (1.12). Then there are vectors $\vec{a}\in\mathcal{A}$ and $\vec{b}\in\mathcal{B}$, such that
\[
\#\{Q\cap S(\vec{a})\cap S(\vec{b}) \}\ll\frac{x}{\log x}\:.
\]
}
The base row of the matrix $\mathcal{M}$ is then defined by 
$$\mathcal{R}=\{(m_0+1)^k+v-1\},\ \text{where}\ 1\leq v\leq y$$
and $m_0$ is defined by $1\leq m_0< P(C_2x)$ and by the congruences
\begin{align*}
&m_0\equiv c_s (\bmod\: s)\\
&m_0\equiv d_p (\bmod\: p) \tag{1.21}\\
&m_0\equiv 0 (\bmod\: q),\ q\in(1, x], q\not\in S\cup P\\
&m_0\equiv e_v (\bmod\: p_v), (e_v, p_v)\ \text{given by}\ v\equiv 1-(e_v+1)^k(\bmod p_v),
\end{align*}
where $e_v\not\equiv -1 (\bmod\: p_v)$, with the possible exceptions of $v$ from an exceptional set $V$ with 
$$\# V\ll x^{1/2+2\epsilon},$$
$$m_0\equiv g_p(\bmod \: p)$$
for all other primes $p\leq C_2x$, $g_p$ arbitrary, but $g_p\not\equiv -1 \bmod\: p$.\\
The matrix $\mathcal{M}$ consists  of the translates of the base row
$$\mathcal{M}= (a_{r, v})_{\substack{1\leq r\leq P(x)^{D-1}\\ 1\leq v\leq y}}$$
with 
$$a_{r,v}=(m_0+1+rP(x))^k+v-1.$$
It can be shown that $a_{r,v}$, $2\leq v\leq y$ is composite unless $v\in V$.\\
The desired $k$-th powers of primes are found in the first column
$$\{a_{r,1}=(m_0+1+rP(x))^k,\ 1\leq r\leq P(x)^{D-1}\}.$$
For the matrix method to work one also needs that $P(x)$ is a good modulus. The definition of this concept will be given in Section 2.\\
This construction now can also be carried out in the case $k=1$ for which it assumes a much simpler form.\\
By change of notation we also may assume that $C_2=1$. In this paper we need the following fact, which we state as:
\begin{lemma}\label{lem11}
There is $m_0$ with $(m_0, P(x))=1$, $1\leq m_0<P(x)$, such that $m_0+v$ is composite for $1\leq v\leq y$.
\end{lemma}

We now come to the second theme: The behaviour of additive arithmetic functions on special sets of integers.\\
An early result was that of Erd\H{o}s and Kac \cite{erdos_kac} who obtained the following:\\
Let $f$ be strongly additive, satisfy $|f(p)|\leq 1$, and let
\[
A(x)=\sum_{p\leq x} \frac{f(p)}{p}\:, \tag{1.22}
\]
\[
B(x)=\sum_{p\leq x}\left( \frac{f(p)^2}{p}\right)^{1/2}\longrightarrow \infty\ \ (x\rightarrow \infty)\:. \tag{1.23}
\]
Then 
\[
x^{-1}\#\left\{ m\::\: \frac{f(m)-A(x)}{B(x)}\leq z\right\}\longrightarrow \frac{1}{\sqrt{2\pi}}\int_{-\infty}^z e^{-w^2/2} dw\ \ (x\rightarrow \infty).  \tag{1.24}
\]
Erd\H{o}s and Kac made essential use of the concept of independent random variables, the Central Limit Theorem and Brun's sieve method.\\
An important example is the number $\omega(m)$ of distinct prime factors for the integer $m$. Here one obtains from (1.20):
\[
x^{-1}\#\left\{ m\::\: \frac{\omega(m)-\log\log m}{\sqrt{\log\log m}}\leq z\right\}\longrightarrow \frac{1}{\sqrt{2\pi}}\int_{-\infty}^z e^{-w^2/2} dw\ \ (x\rightarrow \infty). \tag{1.25}
\]
Later Kubilius \cite{kubilius} gave a different proof of these results, introducing the idea of the Kubilius model, which will also play a role in this paper.\\
For the description of the basic ideas we first sketch a proof of (1.25) by the use of a simple Kubilius model. Later on in Section 2, we use a modified Kubilius model.\\
We cite the model of Kubilius from Elliot \cite{elliott}, p. 119. Let $r$ and $x$ be real numbers with $2\leq r\leq x$. Let 
\[
D=\prod_{\substack{p<x\\ p\ \text{prime}}} p\:, \tag{1.26}
\]
\[
S=\{m\in\mathbb{N}\::\: 1\leq m\leq x\}\:.\tag{1.27}
\]
For each prime $p\mid P$, let
\[
E(p)\ \text{be the set of $m\leq x$ with $p\mid m$}. \tag{1.28}
\]
Let 
$$\overline{E}(p)=S-E(p).$$
For each $k$ which divides $D$ we define the set
\[
E_k=\bigcap_{p\mid k} E(p)\bigcap_{p\mid \frac{D}{k}}\overline{E}(p). \tag{1.29}
\]
We now define two probability measures $\nu$ and $\mu$ on $\mathcal{B}$, and thus obtain the probability spaces
\[
(\mathcal{B}, \nu)\ \ \text{and}\ \ (\mathcal{B}, \mu). \tag{1.30}
\]
The measure $\nu$ is the simple frequency measure. If
$$A=\bigcup_{j=1}^J E_{k_j}\:,$$
then
\[
\nu A=\sum_{j=1}^J [x]^{-1} |E_{k_j}|.  \tag{1.31}
\]
The normal distribution in the result (1.25) suggests to introduce the random variables $X_p$ with 
\begin{equation*}
  X_p(m)=\begin{cases}
    1, & \text{if $p\mid m$},\\
    0, & \text{if $p\nmid m$}
  \end{cases}
\end{equation*}
and the application of the Central Limit Theorem, the independence of the random variables $X_p$, is not fulfilled.\\
With the second probability measure $\mu$ independence of the random variables is given.\\
One defines
\[
\mu E_k= \frac{1}{k} \prod_{p\mid \frac{D}{k}}\left(1-\frac{1}{p}\right). \tag{1.32}
\]
The proof of (1.25) is then completed by a comparison of the measures $\mu$ and $\nu$ and the application of the Central Limit Theorem for the measure $\mu$. The essential estimate for the comparison of $\mu$ and $\nu$ is the estimate
\[
|E_k|=\{1+O(L)\}\frac{x}{k} \prod_{p\mid \frac{D}{k}}\left(1-\frac{1}{p}\right) \tag{1.33}
\]
which is valid, whenever $k$ does not exceed $x^{1/2}$, with 
$$L=\exp\left(-\frac{1}{8}\:\frac{\log x}{\log r}\:\log\left(\frac{\log x}{\log r}\right)\right)+x^{-1/10}.$$
The reader interested in more details might see Elliott \cite{elliott}, pp. 119-123.\\
We conclude the introduction by a statement of our Theorem. First we give the following:
\begin{definition}\label{defn12}
For $m\in\mathbb{N}$, $m\geq 20$ let
\[
\omega^*(m)=\#\left\{p\mid m,\ p\ \text{prime},\ p\geq \frac{\log m}{(\log\log m)^2}\right\}. \tag{1.34}
\]
\end{definition}
\begin{theorem}\label{thm13}
There is a constant $C_3>0$, such that the following holds:\\
Let $\epsilon>0$, $\alpha\in\mathbb{R}$, $u\in\mathbb{N}$. Then there are infinitely many $n$, such that 
$$p_{n+1}-p_n\geq C_3\:\frac{\log p_n \log_2 p_n\log_4 p_n}{\log_3 p_n}$$
and
\[
\frac{\omega^*(p_n+u)-\log\log p_n}{\sqrt{\log\log p_n}}\in (\alpha-\epsilon,\ \alpha+\epsilon). \tag{1.35}
\]
\end{theorem}
\section{Proof of Theorem \ref{thm13}}
We now modify the Kubilius model described in the introduction. First the set $S$ in $(1.27)$ has to be replaced.\\
Let $m_0$ be the number, whose existence has been proven in Lemma \ref{lem11}. Thus we have $m_0$ with $(m_0, P(x))=1$, $1\leq m_0<P(x)$, such that $m_0+v$ is composite for $1\leq v\leq y$.\\
We borrow from \cite{maierpom}, the following definition of a good modulus, which is crucial for the matrix method.
\begin{definition}\label{defn21}
Let us call an integer $q>1$ a good modulus, if $L(s, \chi)\neq 0$ for all characters $\chi\bmod q$ and all $s=\sigma+it$ with
$$\sigma>1-\frac{C_4}{\log(q|t|+1)}\:.$$
This definition depends on the size of $C_4>0$.
\end{definition}
We have
\begin{lemma}\label{lem22}
There is a constant $C_4>0$, such that, in terms of $C_4$, there exist arbitrarily large values of $x$, for which the modulus
$$P(x)=\prod_{p<x} p$$
is good.
\end{lemma}
\begin{proof}
This is Lemma 1 of \cite{maier}.
\end{proof}
In the sequel we assume that $P(x)$ is a good modulus. We now apply the idea sketched in the Introduction: We again choose a set $S^*$ - in analogy to (1.27) and two probability spaces with $\sigma$-algebras and probability measures - in analogy to (1.30).
\begin{definition}\label{defn23}
We set
\begin{align*}
G_0&=(\log x)^2 \tag{2.1}\\
H&= P(x)^{\sigma_1}\ \text{with}\ \sigma_1=\sigma_1(x)=\frac{(\log x)^2}{(\log\log x)^{1/3}}\tag{2.2}\\
\xi^2&=P(x)^{\sigma_2}\ \text{with}\ \sigma_2=\sigma_2(x)=\frac{(\log x)^2}{(\log\log x)^{1/6}}\tag{2.3}.
\end{align*}
\end{definition}
\begin{definition}\label{defn24}
For $r\in\mathbb{N}$, $P(x)^{G_0-1}< r\leq 2P(x)^{G_0-1}$ we write
$$m_1(r)=m_0+rP(x),\ m_2(r)=m_0+rP(x)+u.$$
Let
$$S^*=\{m_2(r)\::\: P(x)^{G_0-1}< r\leq 2P(x)^{G_0-1},\ m_1(r)\ \text{prime}\}.$$
Let
$$D^*=\prod_{p<H}p.$$
For each prime $p$ with $x<p\leq P(x)^{G_1}$ we introduce the set
$$E^*(p)=\{m\in S^*\::\: p\mid m\}.$$
Let 
$$\overline{E^*(p)}=S^*-E^*(p).$$
Corresponding to each integer $k$ which divides $D^*$ we define the set
$$E_k^*=\bigcap_{p\mid k} E^*(p)\bigcap_{p\mid \frac{D^*}{k}} \overline{E^*(p)}.$$
Let $\mathcal{B}$ be the least $\sigma$-algebra which contains the $E^*(p)$. We introduce the measures $\nu^*$ and $\mu^*$. The measure $\nu^*$ is the frequency measure. 
\end{definition}
If
$$A=\bigcup_{j=1}^k E_{k_j}^*\:,$$
then
\[
\nu^*A=\sum_{j=1}^h |S^*|^{-1} |E_{k_j}^*|. \tag{2.4}
\]
We obtain the probability space $(\mathcal{B}, \nu^*).$\\
The random variables $X_p^*$ defined by
\begin{equation*}
  X_p^*(m)=\begin{cases}
    1, & \text{if $p\mid m$},\tag{2.5}\\
    0, & \text{if $p\nmid m$}
  \end{cases}
\end{equation*}
are not independent.\\
The second measure $\mu^*$ is defined by
\[
\mu^*E_k^*=\frac{1}{k}\prod_{p\mid \frac{D^*}{k}} \left(1-\frac{1}{p-1}\right). \tag{2.6}
\]
In the probability space $(\mathcal{B}, \mu^*)$ the random variables $X_p^*$ are independent. We set
\[
\eta^*=\sum_{x<p\leq H} X_p^*. \tag{2.7}
\]
Application of the Central Limit Theorem on the probability space $(\mathcal{B}, \mu^*)$ and the random variable $\eta^*$ gives:
\begin{align*} \tag{2.8}
&\lim_{x\rightarrow \infty} \mu^*\left(\bigcup_{k\mid D^*} E_k^*\::\: \forall r\::\: m(r)\in E_k^*\ |\ \frac{\eta^*(m(r))-\log\log m(r)}{\sqrt{\log\log m(r)}}\in  (\alpha-\epsilon,\ \alpha+\epsilon)\right)\\
&=\frac{1}{\sqrt{2\pi}}\int_{\alpha-\epsilon}^{a+\epsilon} e^{-w^2/2} dw.
\end{align*}
Basic for the comparison of the measure $\mu^*$ with the frequency measure $\nu^*$ is the following.
\begin{lemma}\label{lem25}
Let 
$$\psi(x, q, a)=\sum_{\substack{n\leq x\\ n\equiv a \bmod q}}\Lambda(n).$$
Let $\delta_1, \delta_2, \delta_3$ be arbitrarily small positive constants. Let $R$ be a positive integer,
$$R<\exp\left(\frac{\log Y}{(\log\log Y)^{1+\delta_1}}\right)\:,$$
$Q\geq 1$ and $L=\log YQ$. Assume that $L(s, \chi)\neq 0$ for 
$$Re(s)>1-\frac{\delta_2}{\log(R(|t|+1))}$$
for all $\chi \bmod M$, with $M\leq R^{1+\delta_3}$. Then
\begin{align*}
&\sum_{\substack{q\leq Q\\ (q, R)=1}}\max_{X<Y}\max_{(a, qR)=1}\left|\psi(x, qR, a)-\frac{X}{\phi(qR)}  \right|\\
&\ll_B\frac{Y}{\phi(R)}(\log Y)^{-B}+Y^{1/2}\frac{R^2}{\phi(R)} QL^5\:,
\end{align*}
where $B>0$ is arbitrarily large.
\end{lemma}
\begin{proof}
This is Lemma 10 of \cite{maier2}.
\end{proof}
The determination of $\nu^*E_k^*$ can be formulated as a Sieve problem.\\
We recall some notations from Halberstam-Richert \cite{ref14} (with minor modifications).
\begin{align*}
&\mathcal{Q}=\text{a finite set of integer}\\
&\mathcal{P}=\text{a subset of the set of all primes}\\
&{X^*}=\text{a real number $>1$}\\
&{z}=\text{a real number $\geq 2$}.
\end{align*}
Then we define
$$S(\mathcal{Q}, \mathcal{P}, z)=\left|\{a\in \mathcal{Q}\::\: p\nmid a\ \text{for all}\ p\in \mathcal{P},\ p<z\}\right|.$$
Let $\zeta$ be a multiplicative function, defined for all squarefree positive integers $d$, such that $\zeta(p)=0$, if $p\not\in \mathcal{P}.$ We further define
$$\mathcal{Q}_d=\{a\in \mathcal{Q}\::\: a\equiv 0\bmod d\}\:,$$
$$R_d=|\mathcal{Q}_d|-\frac{\zeta(d)}{d}X$$
$$W(z)=\prod_{p<z}\left(1-\frac{\zeta(p)}{p}\right).$$
Then we have
\begin{lemma}\label{lem26}
Let $\zeta$ satisfy the conditions:
$$(\Omega_1)\::\: 0\leq \frac{\zeta(p)}{p}\leq 1-\frac{1}{A_1}$$
$$(\Omega_2)\::\: \sum_{z'\leq p<z}\frac{\zeta(p)\log p}{p}\leq \kappa\log\frac{z}{z'}+A_2\ \ \text{if}\ 2\leq z'<z,$$
where $\kappa>0$, $A_1\geq 1$, $A_2\geq 1$. Let
$$\xi\geq z, \ \tau=\frac{\log\xi}{\log z},\ \omega(d)=\sum_{p\mid d}1\:.$$
Then 
$$S(\mathcal{Q}, \mathcal{P}, z)=X^* W(z)+Error,$$
with 
\[
Error= O(X^* W(z)\exp(-\tau(\log \tau+1)))+\theta\sum_{\substack{d<\xi^2\\ \zeta(d)\neq 0}} 3^{\omega(d)}|R_d|,
\]
where the constant implied by $``O"$ depends on $\kappa, A_1, A_2$ and $|\theta|\leq 1$.
\end{lemma}
\begin{proof}
This is a special case of (\cite{ref14}, Theorem 7.1, p. 206).
\end{proof}
We now apply Lemma \ref{lem26} with
$$\mathcal{Q}=\mathcal{Q}^{(k)}=\{m\in S^*\::\: m\equiv 0\bmod k\}$$
$$X=\frac{|S^*|}{\phi(k)}, \zeta(p)=\frac{p}{p-1}\:,$$
$$\mathcal{P}=\{p\ \text{prime}\::\: x\leq p<H\},\ z=H$$
$$\xi^2\ \text{from (2.3)},\ \kappa=1$$
and obtain (with $\tau=\frac{\log\xi}{\log H}=(\log\log x)^{1/6})$
\begin{align*}
\nu E_k^*&=\frac{|S^*|}{\phi(k)}\:\prod_{p\mid D^*/k}\left(1-\frac{1}{p-1}\right)\\
&+(1+O(\exp(-\tau(\log\tau+1)))+\theta\sum_{d<\xi^2} 3^{\omega(d)}|R_d^{(k)}|.
\end{align*}
We have
\begin{align*}
\mathcal{Q}_d^{(k)}=&\{m_0+rP(x)+u\::\: m_0+rP(x)\ \text{prime},\ P(x)^{\sigma_0-1}<r\leq 2P(x)^{\sigma_0-1},\\
& m_0+rP(x)+u\equiv 0 \bmod k, m_0+rP(x)+u\equiv 0\bmod d\}\:.
\end{align*}
The two congruences
$$m_0+rP(x)+u\equiv 0 \bmod k$$
and
$$m_0+rP(x)+u\equiv 0 \bmod d$$
are equivalent to a single congruence
$$r\equiv c(k,d)\bmod kd.$$
We thus have (with the familiar notation):
$$\pi(w, q, d)=\#\{ p\leq w,\ p\ \text{prime}, p\equiv d \bmod q\}$$
\begin{align*}
|R_d^{(k)}|&=\left| \mathcal{Q}_d^{(k)}-\frac{|S^*|}{\phi(k)\phi(d)}\right|\\
&\leq\: \vline\bigg(\pi(2P(x)^{\sigma_0}, P(x)kd, c(k,d)-\pi(P(x)^{\sigma_0}, P(x)kd, c(k,d))\\
&\ \ - \left(\frac{li(2P(x)^{\sigma_0}}{\phi(P(x))\phi(kd)}-\frac{li(P(x)^{\sigma_0}}{\phi(P(x))\phi(kd)}\right)\vline
\end{align*}
The sum of errors
\[
\sum_{k\mid D^*}\sum_{d\leq P(x)^{\sigma_2}} R_d^{(k)} \tag{2.10}
\]
can thus be estimated by Lemma \ref{lem25}.\\
The estimate for the sum
$$\sum=\sum_{d<\xi^2} 3^{w(d)}|R_d^{(k)}|$$
may be estimated by
$$\sum\ll \left(\frac{3^{\omega(d)}}{d}|S^*|\right)^{1/2}\left(\sum_{k\mid D^*}\sum_{d\leq P(x)^{\sigma_2}}R_d^{(k)}\right)^{1/2}\:,$$
where we have used the Cauchy-Schwarz inequality and the trivial bound 
$$|R_d|\ll \frac{|S^*|}{d}\:.$$
Thus we obtain
\[
\left|\mu^*\bigcup_{k\mid D^*} E_k^* -\nu^*\bigcup_{k\mid D^*} E_k^*\right| \ll \left|\mu^*\bigcup_{k\mid D^*} E_k^*\right| (\log P(x))^{-B}\:,\tag{2.11}
\]
$B$ arbitrarily large.\\
From (2.8) we have:
\begin{align*}  \tag{2.12}
&\lim_{x\rightarrow\infty} \mu^*\bigcup_{k\mid D^*}\bigg( E_k^*\::\:\forall r\ \text{with}\ m(r)\in E_k^*\::\: \frac{\eta^*(m(r))-\log\log m(r)}{\sqrt{\log\log m(r)}}\:\epsilon(\alpha-\epsilon, \alpha+\epsilon)\\ 
&\ \ \ \ \ \ \ = \frac{1}{\sqrt{2\pi}}\int^{\alpha+\epsilon}_{\alpha-\epsilon} e^{-w^2/2} dw\bigg)\\
&\geq C(\alpha, \epsilon)\lim_{x\rightarrow \infty} \mu^*\left(\bigcup_{k\mid D^*} E_k^*\right)\ \ \text{with\ $C(\alpha, \epsilon)>0$}.
\end{align*}
From (2.11) and (2.12) it follows
\begin{align*}  \tag{2.13}
&\nu^*\bigg(\bigcup_{k\mid D^*}\bigg( E_k^*\::\:\forall r\ \text{with}\ m(r)\in E_k^*\::\: \frac{\eta^*(m(r))-\log\log m(r)}{\sqrt{\log\log m(r)}}\:\epsilon(\alpha-\epsilon+\delta, \alpha+\epsilon-\delta)\\ 
&\ \ \ \ \ \ \ = \frac{1}{\sqrt{2\pi}}\int^{\alpha+\epsilon-\delta}_{\alpha-\epsilon+\delta} e^{-w^2/2} dw\bigg)\\
&\geq C(\alpha,\delta, \epsilon)\mu^*\left(\bigcup_{k\mid D^*} E_k^*\right)\ \ \text{if $x$ is sufficiently large}.
\end{align*}
For the number 
$$\omega^*(m(r))-\eta^*(m(r))$$
of prime factors $p$ of $m(r)$ with 
$$P(x)^{\sigma_1}<p\leq 3P(x)^{\sigma_0}$$
we have
$$P(x)^{\sigma_1(\omega^*(m(r))-\eta^*(m(r)))}\leq 3 P(x)^{\sigma_0}$$
and thus
\[
\omega^*(m(r))-\eta^*(m(r))\ll (\log\log x)^{1/3}.  \tag{2.14}
\]
From (2.13) and (2.14) we obtain Theorem \ref{thm13}.\\

\noindent\textit{Open Problems.} The following questions are of interest:
\begin{enumerate}
\item Can one prove the statement of Theorem \ref{thm13} with the function $\omega$ instead of $\omega^*$, thus also taking into account the prime factors 
$$p<\frac{\log m}{(\log\log m)^2}\:?$$
\item Can one prove a statement involving all the large intervals $(p_n, p_{n+1})$ and not only those obtained by a special construction?
\end{enumerate}

\vspace{3mm}
\textbf{Acknowledgements.} The author wishes to express his gratitude to Professor Helmut Maier for his invaluable insight and the fruitful comments and discussions. He also wishes to thank the referee for his/her very valuable remarks, which helped improve the presentation of the paper.

\end{document}